\documentclass[12pt,tbtags,leqno]{amsart}

\usepackage{amssymb}
\usepackage{amsthm}
\usepackage{amsmath}
\usepackage{verbatim}

\numberwithin{equation}{section}

\newcommand{\HA}{\mathcal{A}}

\newcommand{\D}{\mathbb{D}}

\newcommand{\C}{\mathbb{C}}
\newcommand{\N}{\mathbb{N}}

\newcommand{\T}{\mathbb{T}}

\theoremstyle{plain}
\newtheorem{theorem}{Theorem}[section]
\newtheorem{lemma}[theorem]{Lemma}
\newtheorem{remark}[theorem]{Remark}

\newtheorem{prop}[theorem]{Proposition}
\newtheorem{corollary}[theorem]{Corollary}

\newtheorem{definition}[theorem]{Definition}

\theoremstyle{definition}

\begin{document}

\date{\today}

\author[Shuaibing Luo]{Shuaibing Luo}

\address{Department of Mathematics, University of Tennessee, Knoxville, TN 37996}
\email{luo@math.utk.edu}

\title[Corona Theorem and Stable Rank]
{The corona theorem and Bass stable rank for $M(D(\sum_{i=1}^k a_i \delta_{\zeta_i}))$}
\maketitle

\begin{abstract}
In this paper, we prove the corona theorem for $M(D(\mu_k))$ in two different ways, where $\mu_k = \sum_{i=1}^k a_i \delta_{\zeta_i}$. Then we prove that the Bass stable rank of $M(D(\mu_k))$ is one.

\end{abstract}

\bigskip

\section{Introduction}
Let $\D = \{z \in \C: |z| < 1\}$ be the unit disc. Let $\mu$ be a nonnegative Borel measure on the boundary $\T$ of the unit disc. Let $\varphi_\mu$ be the harmonic function
\begin{equation*}
\varphi_\mu (z) = \int_\T \frac{1-|z|^2}{|\zeta -z|^2} d\mu(\zeta).
\end{equation*}

The Dirichlet type space $D(\mu)$ is defined as the space of all analytic functions on $\D$ such that
\begin{equation*}
\int_\D |f'(z)|^2 \varphi_\mu (z) dA(z)
\end{equation*}
is finite. For any $f \in D(\mu)$, $\|f\|_{D(\mu)}^2 := \|f\|_{H^2(\D)}^2 + \int\limits_\D |f'(z)|^2 \varphi_\mu (z) dA(z)$. When $\mu = \frac{dt}{2\pi}$, $D(\frac{dt}{2\pi})$ is the Dirichlet space $D$.

Dirichlet type spaces were introduced by Richter in \cite{R 91} when studying analytic two-isometries. In \cite{RS 91}, Richter and Sundberg showed that if $f \in D(\delta_\zeta)$, then
\begin{equation*}
D_\zeta (f) = \int_\D |f'(z)|^2 \frac{1-|z|^2}{|\zeta -z|^2} dA(z), \quad \zeta \in \T
\end{equation*}
which is a convenient tool in studying these spaces, where $D_\zeta (f) : = \|\frac{f-f(\zeta)}{z-\zeta}\|_{H^2(\D)}^2$ is called the local Dirichlet integral of $f$ at $\zeta$. Thus, for any $f \in D(\mu)$, $\|f\|_{D(\mu)}^2 = \|f\|_{H^2(\D)}^2 + \int_\T D_\zeta (f) d\mu (\zeta) = \|f\|_{H^2(\D)}^2 + \int_\T \|\frac{f-f(\zeta)}{z-\zeta}\|_{H^2(\D)}^2 d\mu (\zeta)$.

In this paper, we will consider $\mu = \sum _{i=1}^{k} a_i \delta_{\zeta_i}:=\mu_k$, where $a_i$'s are positive numbers, $\zeta_i$'s are in $\T$. Let $M(D(\mu_k))$ be the space of multipliers of $D(\mu_k)$, that is
\begin{align*}
M(D(\mu_k))= \{\phi \in D(\mu_k): \phi f \in D(\mu_k), \forall f \in D(\mu_k)\}.
\end{align*}
Also we will consider $D_{l^2}(\mu_k)$, or $\oplus_1^\infty D(\mu_k)$, which can be considered as $l^2$-valued $D(\mu_k)$ space.

Given $\{\varphi_j\}_{j=1}^\infty \subseteq M(D(\mu_k))$, we let $\Phi(z) = (\varphi_1(z), \varphi_2(z),\ldots)$. We use $M_\Phi$ to denote the (column) operator from $D(\mu_k)$ to $\oplus_1^\infty D(\mu_k)$ defined by
\begin{align*}
M_\Phi (f) = \{\varphi_j f\}_{j=1}^\infty \quad \text{for} ~ f \in D(\mu_k).
\end{align*}

The famous corona theorem goes back to Lennart Carleson. In 1962 Carleson \cite{C 62} proved the absence of a corona in the maximal ideal space of $H^\infty(\D)$ by showing that if $\{\varphi_1,...,\varphi_n\}$ is a finite set of functions in $H^\infty(\D)$ satisfying
\begin{equation}\label{e2 1}
\sum \limits_{j=1}^n |\varphi_j(z)|^2 \geq \eta > 0, \quad z \in \D, \quad (\text{Corona condition}).
\end{equation}
then there are functions $\{f_1,...,f_n\} \subseteq H^\infty(\D)$ with
\begin{equation}\label{e2 2}
\sum \limits_{j=1}^n f_j(z)\varphi_j(z) = 1, \quad z \in \D, \quad (\text{Bezout equation}).
\end{equation}
This is also equivalent to say that the unit disc is dense in the maximal ideal space of $H^\infty (\D)$ in the weak* topology. Then it was shown that the corona theorem is also true in $M(D)$, the multiplier of the Dirichlet space $D$ (see Tolokonnikov \cite{T 91}, Xiao \cite{X 98}). In this paper, we wish to prove the corona theorem for $M(D(\mu_k))$ in two ways. The first version is as follows:

\begin{theorem}\label{corothmuk}
The set of multiplicative linear functionals consisting of evaluations at points of $\D$ is dense in the maximal ideal space of $M(D(\mu_k))$.
\end{theorem}

By the standard Gelfand theory of Banach algebras Theorem \ref{corothmuk} implies:
\begin{corollary}\label{corothmukcor} The following are equivalent:
\begin{itemize}
\item[{\rm (i)}] $\varphi_1,...,\varphi_n \in M(D(\mu_k))$ and there exists a $\eta > 0$ such that
\begin{equation*}
\sum \limits_{j=1}^n |\varphi_j(z)|^2 \geq \eta > 0, \quad z \in \D.
\end{equation*}

\item[{\rm (ii)}] There are functions $b_1,...,b_n \in M(D(\mu_k))$ such that
\begin{equation*}
\sum \limits_{j=1}^n \varphi_j(z)b_j(z) = 1, \quad z \in \D.
\end{equation*}
\end{itemize}
\end{corollary}

Also the corona theorem has been generalized to infinitely many functions in $H^\infty(\D)$ and $M(D)$ (see Rosenblum \cite{R 80}, Tolokonnikov \cite{T 91} and Trent \cite{TT 04}). The infinite version, given by Rosenblum \cite{R 80} and Tolokonnikov \cite{T 91}, can be formulate as follows (see Trent \cite{TT 02}):

\begin{theorem}\label{carlcorona}
Let $\{\varphi_j\}_{j=1}^\infty \subseteq H^\infty(\D)$. Suppose that
\begin{align*}
0 < \epsilon^2 \leq \sum_{j=1} ^\infty |\varphi_j(z)|^2 \leq 1, \quad \text{for all}~ z \in \D.
\end{align*}
Then there exists $\{e_j\}_{j=1}^\infty \subseteq H^\infty(\D)$ such that $\sum_{j=1}^\infty \varphi_j e_j =1$ and $\sup_{z\in \D} \sum_{j=1}^\infty |e_j(z)|^2 \leq \frac{C_0}{\epsilon^2} \ln \frac{1}{\epsilon^2}$, where $C_0$ is a constant.
\end{theorem}

Note that the pointwise hypothesis $\sum_{j=1} ^\infty |\varphi_j(z)|^2 \leq 1$ implies that the operator $T_\Phi$ defined on $H^2(\D)$ in analogy to that of $M_\Phi$ is bounded and $\|T_\Phi\| = \sup_{z\in \D} (\sum_{j=1} ^\infty |\varphi_j(z)|^2)^{\frac{1}{2}}$. Note that since $M(D(\mu_k)) = D(\mu_k) \cap H^\infty(\D)$, the pointwise upper bound hypothesis will not be sufficient to conclude that $M_\Phi$ is bounded from $D(\mu_k)$ to $\oplus_1^\infty D(\mu_k)$. Thus, we will replace the assumption $\sum_{j=1} ^\infty |\varphi_j(z)|^2 \leq 1$ for $z \in \D$ by the condition $\|M_\Phi\| \leq 1$. Then we have the following theorem:

\begin{theorem}\label{d1infin}
Let $\{\varphi_j\}_{j=1}^\infty \subseteq M(D(\mu_k))$. Suppose that
\begin{align*}
\|M_\Phi\| \leq 1 \quad \text{and} \quad 0 < \epsilon^2 \leq \sum_{j=1} ^\infty |\varphi_j(z)|^2 \quad \text{for all}~ z \in \D.
\end{align*}
Then there exists $\{b_j\}_{j=1}^\infty \subseteq M(D(\mu_k))$ such that
\begin{itemize}
\item[{\rm (i)}] $\Phi(z) B(z)^\top = 1$ for all $z \in \D$,
and
\item[{\rm (ii)}] $\|M_B\| \leq \frac{1}{\varepsilon} \Big(2 +16\|M_{B_{k-1}}\|^2 \Big)^{1/2}$, where $B_{k-1}$ is the solution for the corona theorem in $M(D(\mu_{k-1}))$.
\end{itemize}
\end{theorem}

We will use induction to prove Theorem \ref{corothmuk} and Theorem \ref{d1infin}. In section 4, we show that the Bass stable rank of $M(D(\mu_k))$ is one. Throughout this paper, we use $C, C_1, C_2, \ldots$ for absolute constants.

\section{Corona theorem for $M(D(\mu_k))$}

\subsection{}
First, we consider that $k =1$ and $\mu_k = \delta_1$, the unit point mass at $1$. To prove the corona theorem for $M(D(\delta_1))$, we need the following two Lemmas (see \cite{RS 91}).

\begin{lemma}\label{ddel1fuc}
Let $f \in D(\delta_1)$. Then
\begin{itemize}
\item[{\rm (i)}] $f = f(1) + (z-1)g$ for some $g \in H^2(\D)$ and $D_1(f) = \|g\|_{H^(\D)}^2$.
\item[{\rm (ii)}] $\lim_{r\rightarrow 1^-} f(r) = f(1)$.
\item[{\rm (iii)}] $|f(1)| \leq C \|f\|_{D(\delta_1)} ~($see \cite{S 98}$)$.
\end{itemize}

\end{lemma}

\begin{lemma}\label{mulalgduk}
Let $\varphi \in H^\infty(\D)$ and $f \in D(\delta_\zeta)$. Then $\varphi f \in D(\delta_\zeta)$ if and only if $f(\zeta) = 0$ or
$\varphi \in D(\delta_\zeta)$. Furthermore,
$$ D_\zeta(\varphi f) \leq 2 (|| \varphi ||_\infty ^2 D_\zeta(f) + |f(\zeta)|^2 D_\zeta(\varphi))$$
and
$$ |f(\zeta)|^2 D_\zeta(\varphi) \leq 2 (|| \varphi ||_\infty ^2 D_\zeta(f) + D_\zeta(\varphi f)). $$
If $f(\zeta) = 0$ then one even has $D_\zeta(\varphi f) \leq || \varphi ||_\infty ^2 D_\zeta(f)$, while the second inequality can be replaced with the trivial observation that the right-hand side is nonnegative.
\end{lemma}

Thus, by Lemma \ref{mulalgduk}, we have $M(D(\mu_k))= D(\mu_k) \cap H^\infty(\D)$, where $\mu_k = \sum_{i=1}^k a_i \delta_{\zeta_i}$. The norm in $D(\mu_k) \cap H^\infty(\D)$ is defined by
\begin{align*}
||f||_{D(\mu_k) \cap H^\infty(\D)} = ||f||_{D(\mu_k)}+||f||_\infty, \quad f \in D(\mu_k) \cap H^\infty(\D).
\end{align*}

We will use a similar idea as in Lemma 2.1 of \cite{MSW 10} to prove the corona theorem for $M(D(\delta_1))$.

For ease of notation, we let $K: = M(D(\delta_1)) = D(\delta_1)\cap H^\infty(\D)$, and $K_0: = \{f \in K, f(1) = 0\}$. Note that $K_0 \subset K$, and $K_0$ is a Banach algebra without identity.

Note that evaluation at $z \in \D \cup \{1\}$ is a multiplicative linear functional on $K_0$ (if $z =1$, then it is a trivial one). We have the following lemma.

\begin{lemma}\label{mullfk0}
The set of multiplicative linear functionals consisting of evaluations at points of $\D$ is dense in the set of all multiplicative linear functionals on $K_0$.
\end{lemma}

\begin{proof}
Let $m$ be a non-zero multiplicative linear functional on $K_0$, then there exists a function $g_0 \in K_0$, such that $m(g_0) \neq 0$.

If $f \in H^\infty (\D)$, define $M(f) : = \frac{m(fg_0)}{m(g_0)}$.

\textbf{Claim}: $M$ is well-defined, and $M$ is a non-zero multiplicative linear functional on $H^\infty (\D)$.

If we assume that the claim holds, then by Carleson's corona Theorem, there exists a net $(\beta_i)_{i \in I}$ of point evaluations in $\D$ that converges to $M$ in the weak* topology of the maximal ideal space of $H^\infty(\D)$. Note that $m$ is the restriction of $M$ to $K_0$:

\begin{equation*}
M(f) = \frac{m(fg_0)}{m(g_0)} = \frac{m(f) m(g_0)}{m(g_0)} = m(f), f \in K_0.
\end{equation*}
Also the restriction of $(\beta_i)_{i \in I}$ gives a net of point evaluations in $\D$ that converges to $m$ in the weak* topology on the dual space of $K_0$.

We are left to prove the claim: $f \in H^\infty(\D)$, $g_0 \in K_0$, so $f g_0 \in K$ by Lemma \ref{mulalgduk}. Also $(f g_0) (1) = 0$, so $f g_0 \in K_0$, which implies $M$ is well-defined.

Clearly $M$ is linear, when $f \in H^\infty (\D)$,
\begin{align*}
&|M(f)| = |\frac{m(fg_0)}{m(g_0)}| \leq \frac{\|f g_0\|_K}{|m(g_0)|}\\
& = \frac{\|f g_0\|_\infty + \|f g_0\|_{D(\delta_1)}}{|m(g_0)|} \leq \frac{\|f\|_\infty \|g_0\|_\infty + \|f\|_\infty \|g_0\|_{D(\delta_1)}}{|m(g_0)|}\\
& = \frac{\|g_0\|_K}{|m(g_0)|} \|f\|_\infty,
\end{align*}
so $M$ is a bounded functional on $H^\infty(\D)$.

When $f, h \in H^\infty(\D)$, $m(f h g_0) m(g_0) = m (f h g_0 g_0) = m (f g_0) m(h g_0)$, thus we get
\begin{align*}
M(f h) &= \frac{m(f h g_0)}{m(g_0)}\\
& = \frac{[m (f g_0) m(h g_0)] / m(g_0)}{m(g_0)}\\
& = M(f) M(h).
\end{align*}
Therefore the claim is proved.
\end{proof}

Now, we can prove the following Theorem.
\begin{theorem}\label{corothmddel1}
The set of multiplicative linear functionals consisting of evaluations at points of $\D \cup \{1\}$ is dense in the maximal ideal space of $K$.
\end{theorem}

\begin{proof}
Suppose $M$ is a non-zero multiplicative linear functional on $K$.

Let $m = M|_{K_0}$, then $m$ is a multiplicative linear functional on $K_0$. If $f \in K$, then $f - f(1) \in K_0$, so $M(f) = f(1) + m(f - f(1))$.

\underline{{\it Case} 1.} If $m = 0$, then $M(f) = f(1)$, so $M$ is the point evaluation at $1$.

\underline{{\it Case} 2.} If $m \neq 0$, the by Lemma \ref{mullfk0}, there exists a net $(\beta_i)_{i \in I}$ of point evaluations in $\D$ that converges to $m$ in the weak* topology on the dual space of $K_0$. Therefore, for all $f \in K$,
\begin{align*}
M(f) &= f(1) + m(f - f(1)) = f(1) + (\lim \limits_{i \in I} \beta_i) (f - f(1))\\
& = f(1) + \lim \limits_{i \in I} (f(\beta_i) - f(1))\\
& = \lim \limits_{i \in I} f(\beta_i) = (\lim \limits_{i \in I} \beta_i) (f).
\end{align*}
Thus $M = \lim \limits_{i \in I} \beta_i$, and this completes the proof.
\end{proof}

\begin{remark}\label{corthmd1rm}
For any $f \in K, 0 < r < 1$, let $E_r(f) = f(r)$, then from Lemma \ref{ddel1fuc} we have $f(r) \rightarrow f(1)$ as $r \rightarrow 1$. Thus $E_r \rightarrow E_1$ in the weak star topology of $K$ as $r \rightarrow 1$, which implies the set of multiplicative linear functionals consisting of evaluations at points of $\D$ is dense in the maximal ideal space of $K$.
\end{remark}

\subsection{}
In this subsection, we consider general $k \geq 1$. Let $\mu$ be a Borel measure in $\T$ with $\mu(\zeta) = 0$, where $\zeta \in \T$, and suppose that $\D$ is dense in the maximal ideal space of $M(D(\mu))$. Let $H: = M(D(\mu)) \cap D(\delta_\zeta)$  and $H_0: = \{f \in H, f(\zeta) = 0\}$. Then we have:

\begin{lemma}\label{corothmalg}
$H$ is a Banach algebra, $H_0 \subset H$ and $H_0$ is a Banach algebra without identity.
\end{lemma}
\begin{proof}
We only need to verify that $H$ is an algebra. Suppose $f, g \in H = M(D(\mu)) \cap D(\delta_\zeta)$, then $fg \in M(D(\mu))$. Also $f - f(\zeta) \in H$ implies $\frac{f - f(\zeta)}{z - \zeta} g \in H^2(\D)$, thus
\begin{align*}
fg = (z - \zeta) \Big(\frac{f - f(\zeta)}{z - \zeta} g\Big) + f(\zeta) g \in D(\delta_\zeta),
\end{align*}
and so $fg \in H$.
\end{proof}


\begin{lemma}\label{corothmsums}
The set of multiplicative linear functionals consisting of evaluations at points of $\D$ is dense in the maximal ideal space of $H_0$.
\end{lemma}

\begin{proof}
Let $m$ be a non-zero multiplicative linear functional on $H_0$, then there exists a function $g_0 \in H_0$, such that $m(g_0) \neq 0$.

If $f \in M(D(\mu))$, define $M(f) : = \frac{m(fg_0)}{m(g_0)}$.

\textbf{Claim}: $M$ is well-defined, and $M$ is a non-zero multiplicative linear functional on $M(D(\mu))$.

The proof of the claim is similar to the argument in Lemma \ref{mullfk0}. Then there exists a net $(\beta_i)_{i \in I}$ of point evaluations in $\D$ that converges to $M$ in the Gelfand topology of the maximal ideal space of $M(D(\mu))$. Note that $m$ is the restriction of $M$ to $H_0$. Also the restriction of $(\beta_i)_{i \in I}$ gives a net of point evaluations in $\D$ that converges to $m$ in the weak* topology on the dual of $H_0$.

\end{proof}

By the same argument as in Theorem \ref{corothmddel1} we have the following Proposition:
\begin{prop}\label{corothmdu2}
The set of multiplicative linear functionals consisting of evaluations at points of $\D$ is dense in the maximal ideal space of $H$.
\end{prop}

Now we can prove Theorem \ref{corothmuk}.
\begin{proof}
This clearly follows from Proposition \ref{corothmdu2} and induction.
\end{proof}

\begin{remark}\label{corthmsdiri}
If we let $d\mu = \frac{dt}{2\pi}$, then $D(\frac{dt}{2\pi})$ is the Dirichlet space $D$. By Tolokonnikov \cite{T 91}, Xiao \cite{X 98} we have the corona theorem in $M(D)$, then by Proposition \ref{corothmdu2} we also have the corona theorem in $M(D) \cap D(\delta_\zeta)$ for any $\zeta \in \T$.
\end{remark}

\section{Infinite version for $M(D(\mu_k))$}
\subsection{}
First, we consider $M(D(\delta_1))$.

The following Lemma can be derived from \cite[Lemma 6]{TT 04} (see also \cite{RT 11}).
\begin{lemma}\label{koszmats}
Let $\{a_j\}_{j=1}^\infty \in l^2$ and $A = (a_1, a_2, \ldots) \in B(l^2,\C)$. Then there exists an $\infty \times \infty$ matrix $Q_A$, such that the entries of $Q_A$ belong to the set $\{0, \pm a_j: j =1,2,\ldots\}$ and $Q_A$ satisfies
\begin{itemize}
\item[{\rm (a)}] range of $Q_A \subseteq$ kernel of $A$.

\item[{\rm (b)}] $(AA^*)I - A^*A = Q_A Q_A^*$.

\item[{\rm (c)}] If $\{d_j\}_{j=1}^\infty \in l^2$ and $D = (d_1, d_2, \ldots)$, then
\begin{align*}
(AD^\top)I - D^\top A = Q_A Q_D^\top.
\end{align*}
\end{itemize}
\end{lemma}

We need one lemma before we prove the corona theorem for infinitely many functions in $M(D(\delta_1))$.
\begin{lemma}\label{conat1no0}
Let $\{\varphi_j\}_{j=1}^\infty \subseteq M(D(\delta_1))$. Then
\begin{itemize}
\item[{\rm (i)}] $M_\Phi$ is a bounded operator if and only if $\sum_{j=1} ^\infty \|\varphi_j\|_{D(\delta_1)}^2$ and $\sup_{z \in \D}\sum_{j=1} ^\infty |\varphi_j(z)|^2$ are finite.

\item[{\rm (ii)}] If $\|M_\Phi\| \leq 1$ and $0 < \epsilon^2 \leq \sum_{j=1} ^\infty |\varphi_j(z)|^2$ for all $z \in \D$, then
\begin{align*}
\Phi(1) = (\varphi_1(1), \varphi_2(1), \ldots) \neq 0.
\end{align*}

\item[{\rm (iii)}] If $\|M_\Phi\| \leq 1$ and $f = \sum_{i=1}^\infty[\varphi_i - \varphi_i(1)] \overline{\varphi_i(1)}$, then $f \in M(D(\delta_1))$ and $f(1) = 0$.
\end{itemize}
\end{lemma}

\begin{proof}
(i): Suppose that $M_\Phi$ is bounded from $D(\delta_1)$ to $\oplus_1^\infty D(\delta_1)$ with $\|M_\Phi\| \leq 1$, then $\sup_{z \in \D}\sum_{j=1} ^\infty |\varphi_j(z)|^2 \leq 1$ (see \cite{TT 04}). Let $f =1 \in D(\delta_1)$, then
\begin{align*}
\sum_{j=1} ^\infty \|\varphi_j\|_{D(\delta_1)}^2 &= \|M_\Phi f\|_{\oplus_1^\infty D(\delta_1)}^2\\
&\leq \|M_\Phi\|^2 \|1\|_{D(\delta_1)} \leq 1.
\end{align*}

Conversely suppose $\sup_{z \in \D}\sum_{j=1} ^\infty |\varphi_j(z)|^2 \leq 1$ and $\sum_{j=1} ^\infty \|\varphi_j\|_{D(\delta_1)}^2 \leq 1$. Let $f \in D(\delta_1)$, suppose $f = f(1) + (z-1)g$ for some $g \in H^2(\D)$, then $D_1 (f) = \|g\|_{H^2(\D)}^2$ and
\begin{align*}
&\|M_\Phi f\|_{\oplus_1^\infty D(\delta_1)}^2 = \sum_{j=1} ^\infty \|\varphi_j f\|_{D(\delta_1)}^2\\
& = \sum_{j=1} ^\infty \|\varphi_j f\|_{H^2(\D)}^2 + \sum_{j=1} ^\infty \|\frac{\varphi_j f - (\varphi_j f)(1)}{z-1}\|_{H^2(\D)}^2\\
&\leq \|f\|_{H^2(\D)}^2 + \sum_{j=1} ^\infty \Big[2\|\frac{\varphi_j f(1) - (\varphi_j f)(1)}{z-1}\|_{H^2(\D)}^2 + 2\|\frac{\varphi_j g (z-1)}{z-1}\|_{H^2(\D)}^2\Big]\\
&\leq \|f\|_{H^2(\D)}^2 + 2 |f(1)|^2 \sum_{j=1} ^\infty D_1 (\varphi_j) + 2 \|g\|_{H^2(\D)}^2\\
& \leq 2 \|f\|_{D(\delta_1)} + 2 |f(1)|^2.
\end{align*}
Since $|f(1)| \leq C \|f\|_{D(\delta_1)}$ (see \cite{S 98}), we conclude that $M_\Phi$ is bounded from $D(\delta_1)$ to $\oplus_1^\infty D(\delta_1)$.

(ii): Suppose $\{g_j\}_{j=1}^\infty \subseteq H^2(\D)$ such that
\begin{align*}
\varphi_j (z) = \varphi_j (1) + (z-1) g_j(z), \quad \text{and} \quad D_1(\varphi_j) = \|g_j\|_{H^2(\D)}^2, j = 1, 2,\cdots.
\end{align*}
Note that
\begin{align*}
|\varphi_j(z)|^2 &\leq |\varphi_j (1)|^2 + |z-1|^2  |g_j(z)|^2 + 2|\varphi_j (1)||z-1|  |g_j(z)|\\
&\leq (1+\eta) |\varphi_j (1)|^2 + (1+\frac{1}{\eta})|z-1|^2  |g_j(z)|^2,
\end{align*}
where $\eta$ is any positive number. Then we have
\begin{align*}
\epsilon^2 &\leq \sum_{j=1} ^\infty |\varphi_j(z)|^2\leq \sum_{j=1} ^\infty (1+\eta) |\varphi_j (1)|^2 + (1+\frac{1}{\eta})|z-1|^2  |g_j(z)|^2\\
&\leq \sum_{j=1} ^\infty (1+\eta) |\varphi_j (1)|^2 + (1+\frac{1}{\eta}) \frac{|z-1|^2}{1-|z|^2} \sum_{j=1} ^\infty \|\varphi_j\|_{D(\delta_1)}^2\\
&\leq \sum_{j=1} ^\infty (1+\eta) |\varphi_j (1)|^2 + (1+\frac{1}{\eta}) \frac{|z-1|^2}{1-|z|^2} \quad \text{for all} ~ z \in \D,
\end{align*}
where in the last inequality we used part (i). Let $z = r \rightarrow 1^-$ we get
\begin{align*}
\epsilon^2 &\leq \sum_{j=1} ^\infty (1+\eta) |\varphi_j (1)|^2 := (1+\eta) |\Phi(1)|^2.
\end{align*}
Let $\eta \rightarrow 0$, we have $|\Phi(1)|^2 = \sum_{j=1} ^\infty |\varphi_j (1)|^2 \geq \epsilon^2$, thus $\Phi(1) = (\varphi_1(1), \varphi_2(1), \ldots) \neq 0$.

(iii) Suppose $\|M_\Phi\| \leq 1$ and $f = \sum_{i=1}^\infty[\varphi_i - \varphi_i(1)] \overline{\varphi_i(1)}$, then $f \in H^\infty (\D)$ and
\begin{align*}
\|f\|_{D(\delta_1)}^2 &= \|\sum_{i=1}^\infty[\varphi_i - \varphi_i(1)] \overline{\varphi_i(1)}\|_{D(\delta_1)}^2\\
&\leq \sum_{i=1}^\infty \|\varphi_i - \varphi_i(1)\|_{D(\delta_1)}^2 \sum_{i=1}^\infty |\varphi_i(1)|^2\\
&\leq 2 \Big[\sum_{i=1}^\infty \|\varphi_i\|_{D(\delta_1)}^2 + \sum_{i=1}^\infty |\varphi_i(1)|^2\Big] \sum_{i=1}^\infty |\varphi_i(1)|^2\\
&\leq 4,
\end{align*}
where in the last inequality we used part (i).

For any $k \in \N$, let $f_k = \sum_{i=1}^k [\varphi_i - \varphi_i(1)] \overline{\varphi_i(1)}$. Then $f_k \rightarrow f \in D(\delta_1)$, note that $f_k(1) = 0$ and point evaluation at $1$ is continuous, we conclude that $f(1) = 0$.
\end{proof}


Now we can prove the corona theorem for $M(D(\delta_1))$.

\begin{theorem}\label{d1infinp1}
Let $\{\varphi_j\}_{j=1}^\infty \subseteq M(D(\delta_1))$. Suppose that $\|M_\Phi\| \leq 1$ and $0 < \epsilon^2 \leq \sum_{j=1} ^\infty |\varphi_j(z)|^2$ for all $z \in \D$. Then there exists $\{b_j\}_{j=1}^\infty \subseteq M(D(\delta_1))$ such that
\begin{itemize}
\item[{\rm (i)}] $\Phi(z) B(z)^\top = 1$ for all $z \in \D$,
and
\item[{\rm (ii)}] $\|M_B\| \leq \frac{1}{\varepsilon} (2 + 8\frac{C_0}{\epsilon^2} \ln \frac{1}{\epsilon^2})^{1/2}$.
\end{itemize}
\end{theorem}

\begin{proof}
(i): By Theorem \ref{carlcorona}, there exists an $E \in H^\infty_{l^2}(\D)$ such that
\begin{align*}
\Phi(z) E(z)^\top = 1 \quad \text{for} ~z \in \D,
\end{align*}
and
\begin{align*}
\|E\|_{H^\infty_{l^2}(\D)}^2 := \sup_{z \in \D} \sum_{j=1} ^\infty |e_j(z)|^2 \leq \frac{C_0}{\epsilon^2} \ln \frac{1}{\epsilon^2}.
\end{align*}

Let $A = \Phi(z), D = E(z)$ in Lemma \ref{koszmats}, then
\begin{align*}
I - E(z)^\top \Phi(z) = Q_{\Phi(z)} Q_{E(z)}^\top,
\end{align*}
thus
\begin{align}\label{dlincorequ}
I = E(z)^\top \Phi(1) +  E(z)^\top (\Phi(z) - \Phi(1)) + Q_{\Phi(z)} Q_{E(z)}^\top.
\end{align}
Let $\Phi(1)^* = (\overline{\varphi_1(1)}, \overline{\varphi_2(1)}, \ldots)^\top$, then $|\Phi(1)|^2 = \Phi(1) \Phi(1)^*$ and
\begin{align}
\Phi(1)^* = E(z)^\top |\Phi(1)|^2 +  E(z)^\top [\Phi(z) - \Phi(1)]\Phi(1)^* \label{d1infkos}\\
\hspace{5cm}+ Q_{\Phi(z)} Q_{E(z)}^\top \Phi(1)^* \notag.
\end{align}

By Lemma \ref{conat1no0} we have $\Phi(1) = (\varphi_1(1), \varphi_2(1), \ldots) \neq 0$, then from (\ref{d1infkos}) we have
\begin{align*}
\frac{\Phi(1)^*}{|\Phi(1)|^2} = E(z)^\top +  E(z)^\top \frac{[\Phi(z) - \Phi(1)] \Phi(1)^*}{|\Phi(1)|^2}+ Q_{\Phi(z)} Q_{E(z)}^\top \frac{\Phi(1)^*}{|\Phi(1)|^2},
\end{align*}
therefore,
\begin{align*}
&E(z)^\top + Q_{\Phi(z)} Q_{E(z)}^\top \frac{\Phi(1)^*}{|\Phi(1)|^2} = \frac{\Phi(1)^*}{|\Phi(1)|^2} - \frac{[\Phi(z) - \Phi(1)] \Phi(1)^*}{|\Phi(1)|^2} E(z)^\top\\
& = \frac{\Phi(1)^*}{|\Phi(1)|^2} - \frac{\sum_{i =1}^\infty [\varphi_i(z) - \varphi_i(1)] \overline{\varphi_i(1)}}{|\Phi(1)|^2} E(z)^\top.
\end{align*}

Let $B(z)^\top = E(z)^\top + Q_{\Phi(z)} Q_{E(z)}^\top \frac{\Phi(1)^*}{|\Phi(1)|^2}$. From Lemma \ref{koszmats}, we have
\begin{align*}
\Phi(z) B(z)^\top = 1 \quad \text{for} ~z \in \D,
\end{align*}
and
$$b_j (z) = \frac{\overline{\varphi_j(1)}}{|\Phi(1)|^2} - \frac{\sum_{i =1}^\infty [\varphi_i(z) - \varphi_i(1)] \overline{\varphi_i(1)}}{|\Phi(1)|^2} e_j(z), j = 1, 2, 3, \cdots.$$

By Lemma \ref{conat1no0} we have $f:=\sum_{i =1}^\infty [\varphi_i - \varphi_i(1)] \overline{\varphi_i(1)} \in M(D(\delta_1))$ and $f(1) = 0$. Thus from Lemma \ref{mulalgduk} we have $b_j \in H^\infty (\D) \cap D(\delta_1) = M (D(\delta_1)), j = 1, 2,\cdots$.

(ii): Let $f \in D(\delta_1)$, then
\begin{align*}
&\sum_{j=1} ^\infty \|b_j f\|_{D(\delta_1)}^2 \\
&\leq \frac{2}{|\Phi(1)|^4} \Big[\sum_{j=1} ^\infty \|\overline{\varphi_j(1)} f\|_{D(\delta_1)}^2 + \sum_{j=1} ^\infty \|\sum_{i =1}^\infty [\varphi_i - \varphi_i(1)] \overline{\varphi_i(1)} e_j f\|_{D(\delta_1)}^2\Big]\\
&\leq \frac{2}{|\Phi(1)|^4} \Big[|\Phi(1)|^2 \|f\|_{D(\delta_1)}^2 + \sum_{j=1} ^\infty \sum_{i =1}^\infty \|[\varphi_i - \varphi_i(1)] e_j f\|_{D(\delta_1)}^2 |\Phi(1)|^2 \Big]\\
& = \frac{2}{|\Phi(1)|^2} \Big[\|f\|_{D(\delta_1)}^2 + \sum_{j=1} ^\infty \sum_{i =1}^\infty \|[\varphi_i - \varphi_i(1)] e_j f\|_{D(\delta_1)}^2\Big]
\end{align*}

Note that
\begin{align}\label{coressumff}
&\sum_{j=1} ^\infty \sum_{i =1}^\infty \|[\varphi_i - \varphi_i(1)] e_j f\|_{D(\delta_1)}^2\\
& = \sum_{j=1} ^\infty \sum_{i =1}^\infty \|[\varphi_i - \varphi_i(1)] e_j f\|_{H^2(\D)}^2 + \sum_{j=1} ^\infty \sum_{i =1}^\infty \|\frac{\varphi_i - \varphi_i(1)}{z-1} e_j f\|_{H^2(\D)}^2 \notag\\
&\leq \|E\|_{H^\infty_{l^2}(\D)}^2 \sum_{i =1}^\infty \Big[\|(\varphi_i - \varphi_i(1)) f\|_{H^2(\D)}^2 + \|\frac{\varphi_i - \varphi_i(1)}{z-1} f\|_{H^2(\D)}^2 \Big] \notag\\
&= \|E\|_{H^\infty_{l^2}(\D)}^2 \sum_{i =1}^\infty \|(\varphi_i - \varphi_i(1)) f\|_{D(\delta_1)}^2 \notag\\
&\leq 2\|E\|_{H^\infty_{l^2}(\D)}^2 \Big[\sum_{i =1}^\infty \|\varphi_i f\|_{D(\delta_1)}^2 + \sum_{i =1}^\infty \|\varphi_i(1) f\|_{D(\delta_1)}^2\Big] \notag\\
&\leq 2 \|E\|_{H^\infty_{l^2}(\D)}^2 \Big[\|M_\Phi\|^2 + |\Phi(1)|^2\Big] \|f\|_{D(\delta_1)}^2 \notag\\
&\leq 4 \|E\|_{H^\infty_{l^2}(\D)}^2 \|f\|_{D(\delta_1)}^2. \notag
\end{align}
Thus
\begin{align*}
\sum_{j=1} ^\infty \|b_j f\|_{D(\delta_1)}^2 \leq \frac{2}{|\Phi(1)|^2} \Big[\|f\|_{D(\delta_1)}^2 + 4 \|E\|_{H^\infty_{l^2}(\D)}^2 \|f\|_{D(\delta_1)}^2\Big],
\end{align*}
therefore
\begin{align*}
\|M_B\| &\leq \Big[\frac{2}{|\Phi(1)|^2} (1 + 4 \|E\|_{H^\infty_{l^2}(\D)}^2) \Big]^{1/2}\\
& \leq \frac{1}{\varepsilon} (2 + 8\frac{C_0}{\epsilon^2} \ln \frac{1}{\epsilon^2})^{1/2},
\end{align*}
where in the last inequality we used $|\Phi(1)| \geq \varepsilon$ in the proof of Lemma \ref{conat1no0}.
\end{proof}

\begin{remark}\label{corthmbddsol}
From equation (\ref{dlincorequ}), we can get another corona solution $D(z) = (d_1(z), d_2(z),\cdots,)$ such that
\begin{align}\label{bezeqinfmf}
\sum_{j=1}^\infty \varphi_j(z) d_j(z) = 1, \quad z \in \D.
\end{align}
Suppose $|\varphi_{1}(1)| = \max_{\{j=1,2,\ldots\}} |\varphi_{j}(1)|$, let $d_1 (z) = \frac{1}{\varphi_1(1)} - \frac{\varphi_1(z) - \varphi_1(1)}{\varphi_1(1)} e_1(z)$, $d_j (z) = - \frac{\varphi_1(z) - \varphi_1(1)}{\varphi_1(1)} e_j(z), j = 2, 3, \cdots$. Then (\ref{bezeqinfmf}) is satisfied and we have
$$\|M_D\| \leq \Big[\frac{2}{|\varphi_1(1)|^2} + 4 \big(\frac{\|\varphi_1\|_{M(D(\delta_1))}^2}{|\varphi_1(1)|^2} + 1\big) \frac{C_0}{\epsilon^2} \ln \frac{1}{\epsilon^2} \Big]^{1/2},$$
but in this case the bound of the corona solution depends on the chosen $\varphi_1$. It would be of interested to determine the best possible
bound for the solution $B$ in terms of $\|M_\Phi\|$ and $\varepsilon$.
\end{remark}

\subsection{}
For general $k$, we use induction to prove Theorem \ref{d1infin}.
\begin{proof}
The idea is the same as in Theorem \ref{d1infinp1}. We sketch a proof here.

If $k=1$, then by Theorem \ref{d1infinp1}, it is true.

Suppose $k=l \geq 1$, it is true.

If $k=l+1$, note that $\{\varphi_j\}_{j=1}^\infty \subseteq M(D(\mu_{l+1})) \subseteq M(D(\mu_{l}))$,
by induction, there exists $\{e_j\}_{j=1}^\infty \subseteq M(D(\mu_{l}))$ such that
\begin{align*}
\Phi(z) E(z)^\top = 1 \quad \text{for} ~z \in \D,
\end{align*}
and
\begin{align*}
\|M_E\| \leq \frac{1}{\varepsilon} \Big(2 +16\|M_{B_{l-1}}\|^2 \Big)^{1/2},
\end{align*}

Following the same argument as in Lemma \ref{conat1no0}, we have $\Phi(\zeta_{l+1}) = (\varphi_1(\zeta_{l+1}), \varphi_2(\zeta_{l+1}),\ldots) \neq 0$ and
\begin{align}\label{dfininfkos}
I = E(z)^\top \Phi(\zeta_{l+1}) +  E(z)^\top (\Phi(z) - \Phi(\zeta_{l+1})) + Q_{\Phi(z)} Q_{E(z)}^\top.
\end{align}

Thus
$$b_j(z) = \frac{\overline{\varphi_j(\zeta_{l+1})}}{|\Phi(\zeta_{l+1})|^2} - \frac{\sum_{i =1}^\infty [\varphi_i(z) - \varphi_i(\zeta_{l+1})] \overline{\varphi_i(\zeta_{l+1})}}{|\Phi(\zeta_{l+1})|^2} e_j(z) \in M(D(\mu_l)),$$
and $\Phi(z) B(z)^\top = 1$ for all $z \in \D$.

Now we estimate $\|M_B\|$. Let $f \in D(\mu_{l+1})$, then
\begin{align*}
&\sum_{j=1} ^\infty \|b_j f\|_{D(\mu_{l+1})}^2\\
&\leq \frac{2}{|\Phi(\zeta_{l+1}))|^2} \Big[\|f\|_{D(\mu_{l+1})}^2 + \sum_{j=1} ^\infty \sum_{i =1}^\infty \|[\varphi_i - \varphi_i(\zeta_{l+1}))] e_j f\|_{D(\mu_{l+1})}^2\Big].
\end{align*}

Suppose $\mu_{l+1} = \mu_l + \delta_{\zeta_{l+1}}$, note that using inequality (\ref{coressumff}) we have
\begin{align*}
&\sum_{j=1} ^\infty \sum_{i =1}^\infty\|[\varphi_i - \varphi_i(\zeta_{l+1}))] e_j f\|_{D(\mu_{l+1})}^2\\
& \leq \sum_{j=1} ^\infty \sum_{i =1}^\infty\|[\varphi_i - \varphi_i(\zeta_{l+1}))] e_jf\|_{D(\mu_{l})}^2\\
&\hspace{3cm} + \sum_{j=1} ^\infty \sum_{i =1}^\infty\|[\varphi_i - \varphi_i(\zeta_{l+1}))] e_j f\|_{D(\delta_{\zeta_{l+1}})}^2\\
& \leq \sum_{i=1} ^\infty \|M_E\|^2 \|[\varphi_i - \varphi_i(\zeta_{l+1}))] f\|_{D(\mu_{l})}^2 + 4 \|E\|_{H^\infty_{l^2}(\D)}^2 \|f\|_{D(\delta_{\zeta_{l+1}})}^2\\
&\leq \|M_E\|^2 2\Big[\|M_\Phi\| + |\Phi(\zeta_{l+1})|^2) \Big] \|f\|_{D(\mu_{l+1})}^2 + 4 \|E\|_{H^\infty_{l^2}(\D)}^2 \|f\|_{D(\delta_{\zeta_{l+1}})}^2\\
&\leq 4 \|M_E\|^2 \|f\|_{D(\mu_{l+1})}^2 + 4 \|M_E\|^2 \|f\|_{D(\mu_{l+1})}^2\\
&= 8 \|M_E\|^2 \|f\|_{D(\mu_{l+1})}^2.
\end{align*}

Thus
\begin{align*}
\sum_{j=1} ^\infty \|b_j f\|_{D(\mu_{l+1})}^2 & \leq \frac{2}{|\Phi(\zeta_{l+1}))|^2} \Big[\|f\|_{D(\mu_{l+1})}^2 + 8 \|M_E\|^2 \|f\|_{D(\mu_{l+1})}^2 \Big]\\
& \leq \frac{1}{\varepsilon^2} \Big(2 +16\|M_E\|^2 \Big) \|f\|_{D(\mu_{l+1})}^2,
\end{align*}
and so $\|M_B\| \leq \frac{1}{\varepsilon} \Big(2 +16\|M_E\|^2 \Big)^{1/2}$.
\end{proof}

\section{Bass stable rank for $M(D(\sum_{i=1}^k a_i \delta_{\zeta_i}))$}
The notion of stable rank of a ring was introduced by Bass \cite{B 64} to facilitate computations in algebraic K-theory. Let us recall the main definition.
\begin{definition}
Let $\HA$ be any ring with identity 1. An n-tuple $a = (a_1, \ldots, a_n) \in \HA^n$ is called unimodular or invertible, if there exists an n-tuple $b = (b_1, \ldots, b_n) \in \HA^n$ such that $\sum_{i=1}^n a_i b_i =1$. The set of all invertible n-tuples is denoted by $U_n(\HA)$. An (n+1)-tuple $x = (x_1, \ldots, x_{n+1}) \in \HA^{n+1}$ is called reducible, if there exists an n-tuple $y = (y_1, \ldots, y_n)$ such that $(x_1 + y_1 x_{n+1}, \ldots, x_n + y_n x_{n+1})$ is invertible. The Bass stable rank of $\HA$ is the least n such that every invertible (n+1)-tuple is reducible.
\end{definition}

In recent years, the Bass stable rank has been studied by many authors in the setting of Banach algebras. Jones, Marshall and Wolff \cite{JMW 86} showed that the Bass stable rank of the disc algebra $A(\D)$ is one; Treil \cite{Tr 92} proved that the Bass stable rank of $H^\infty(\D)$ is one; and in \cite{MSW 10}, Mortini, Sasane, and Wick showed that the Bass stable rank of $\C+BH^\infty$ and $A_B$ are one as well. In this paper, we show that the Bass stable rank of $M(D(\mu_k))$ is also one, where $\mu_k = \sum_{i=1}^k a_i \delta_{\zeta_i}$.

First, we prove that the Bass stable rank of $M(D(\delta_1)) = D(\delta_1) \cap H^\infty(\D)$ is one.

\begin{lemma}\label{bsrdd1}
The Bass stable rank of $D(\delta_1) \cap H^\infty(\D)$ is one.
\end{lemma}

\begin{proof}
Let $(f, h)$ be a unimodular pair in $(D(\delta_1) \cap H^\infty(\D))^2$, i.e., there exists $(g_1, g_2) \in (D(\delta_1) \cap H^\infty(\D))^2$ such that $f g_1 + h g_2 = 1$. Then $\inf_{z\in\D} |f(z)| + |h(z)|:=\eta>0$.

\underline{{\it Case} 1.} If $f(1) \neq 0$, then we claim $(f, (f-f(1))h)$ is unimodular.

In fact, if $z\in\D$ is such that $|f(z) -f(1)| \geq \frac{|f(1)|}{2}$, then $|f(z)| + |(f(z)-f(1)h(z)| \geq |f(z)| + \frac{|f(1)|}{2} |h(z)| \geq \min\{1, \frac{|f(1)|}{2}\} \eta$.

If $z\in\D$ is such that $|f(z) -f(1)| \leq \frac{|f(1)|}{2}$, then $|f(z)| = |f(z) - f(1) + f(1)| \geq |f(1)| - |f(z) - f(1)| \geq |\frac{|f(1)|}{2}|$, and so $|f(z)| + |(f(z)-f(1)h(z)| \geq |f(z)| \geq |\frac{|f(1)|}{2}|$.

Thus, $(f, (f-f(1))h)$ is unimodular. By Theorem 1 in \cite{Tr 92}, there is some element $g \in H^\infty(\D)$ such that $f + g [(f-f(1))h]$ is invertible in $H^\infty(\D)$. Note that $g(f-f(1)) \in D(\delta_1) \cap H^\infty(\D)$, by the corona theorem for $M(D(\delta_1))$, we get that $f + g [(f-f(1))h]$ is also invertible in $D(\delta_1) \cap H^\infty(\D)$.

\underline{{\it Case} 2.} If $f(1) = 0$, then $h(1) \neq 0$, since $\inf_{z\in\D} |f(z)| + |h(z)|:=\eta>0$. We claim the pair $(f+h,h)$ is unimodular: By the corona theorem for $M(D(\delta_1))$, there exists $(g_1, g_2) \in (D(\delta_1) \cap H^\infty(\D))^2$ such that $f g_1 + h g_2 = 1$, so $(f+h) g_1 + h (g_2 - g_1) = 1$, which implies $(f+h,h)$ is unimodular.

By Case 1, there exists some $g \in D(\delta_1) \cap H^\infty(\D)$, such that $(f+h) + gh$ is invertible in $D(\delta_1) \cap H^\infty(\D)$. Note that $(f+h) +gh= f + (1+g)h$, and $1+g \in D(\delta_1) \cap H^\infty(\D)$, we are done.

\end{proof}

Now we show the Bass stable rank of $M(D(\delta_{\zeta_1}) \cap D(\delta_{\zeta_2})) = D(\delta_{\zeta_1}) \cap D(\delta_{\zeta_2}) \cap H^\infty(\D)$ is one.
\begin{lemma}\label{bsrdd2}
The Bass stable rank of $D(\delta_{\zeta_1}) \cap D(\delta_{\zeta_2}) \cap H^\infty(\D)$ is one.
\end{lemma}
\begin{proof}
Let $(f, h)$ be a unimodular pair in $(D(\delta_{\zeta_1}) \cap D(\delta_{\zeta_2}) \cap H^\infty(\D))^2$.

\underline{{\it Case} 1.}  $f(\zeta_2) \neq 0$. As in Lemma \ref{bsrdd1} we conclude that $(f, (f-f(\zeta_2))h)$ is unimodular. Then by Lemma \ref{bsrdd1}, there exists some $g \in D(\delta_1) \cap H^\infty(\D)$ such that $f + g[(f-f(\zeta_2))h]$ is invertible in $D(\delta_1) \cap H^\infty(\D)$. Note that $g (f-f(\zeta_2)) \in D(\delta_{\zeta_1}) \cap D(\delta_{\zeta_2}) \cap H^\infty(\D)$, by the corona theorem for $M(D(\delta_{\zeta_1}) \cap D(\delta_{\zeta_2}))$, we get $f + g [(f-f(1))h]$ is also invertible in $D(\delta_{\zeta_1}) \cap D(\delta_{\zeta_2}) \cap H^\infty(\D)$.

\underline{{\it Case} 2.} $f(\zeta_2) = 0$. As in Lemma \ref{bsrdd1}, we consider the pair $(f+h,h)$ and conclude that the Bass stable rank of $D(\delta_{\zeta_1}) \cap D(\delta_{\zeta_2}) \cap H^\infty(\D)$ is one.
\end{proof}

For general $k$, by induction we obtain that the Bass stable rank of $M(D(\mu_k))$ is one.

\begin{theorem}\label{basssrduk}
The Bass stable rank of $M(D(\mu_k))$ is one.
\end{theorem}

\vspace{1cm}
\noindent \textbf{Acknowledges.}
This work appeared as part of the author's doctoral dissertation at the University of Tennessee under the supervision of Dr. Stefan Richter. The author would like to thank Dr. Richter for many useful conversations.


\vspace{1cm}


\begin{thebibliography}{99}
\bibitem{B 64}
H. Bass, K-theory and stable algebra, Inst. Hautes \'{E}tudes Sci. Publ. Math. No. 22, 5-60 (1964).

\bibitem{C 62}
L. Carleson, Interpolations by bounded analytic functions and the corona problem, Ann. of Math. (2) 76, 547-559 (1962).

\bibitem{JMW 86}
P. W. Jones, D. Marshall, T. Wolff, Stable rank of the disc algebra, Proc. Amer. Math. Soc. 96 , no. 4, 603-604 (1986).

\bibitem{MSW 10}
R. Mortini, A. Sasane, B. D. Wick, The corona theorem and stable rank for the algebra $\C+BH^\infty$, Houston J. Math. 36, no. 1, 289-302 (2010).

\bibitem{R 91}
S. Richter, A representation theorem for cyclic analytic two-isometries, Trans. Amer. Math. Soc. 328, no. 1, 325-349 (1991).

\bibitem{RS 91}
S. Richter, C. Sundberg, A formula for the local Dirichlet integral, Michigan Math. J. 38, 355-379 (1991).

\bibitem{R 80}
M. Rosenblum, A corona theorem for countably many functions, Integral Equations and Operator Theory 3, 125-137 (1980).

\bibitem{RT 11}
J. Ryle, T. T. Trent, A corona theorem for certain subalgebras of $H^\infty(\D)$, Houston J. Math. 37, no. 4, 1211-1226 (2011)

\bibitem{S 02}
S. Shimorin, Complete Nevanlinna-Pick property of Dirichlet-type spaces, J. Funct. Anal. 191 , no. 2, 276-296 (2002).

\bibitem{T 91}
V. A. Tolokonnikov, The corona theorem in algebras of bounded analytic functions, Amer. Math. Soc. Trans. 149, 61-93 (1991).

\bibitem{S 98}
D, Sarason, Harmonically weighted Dirichlet spaces associated with finitely atomic measures. Integral Equ. Oper. Theory 31, 186-213 (1998).

\bibitem{Tr 92}
S. Treil, The stable rank of the algebra $H^\infty$ equals 1, J. Funct. Anal. 109 , no. 1, 130-154 (1992).

\bibitem{TT 04}
T. T. Trent, A corona theorem for multipliers on Dirichlet space. Integral Equations Operator Theory 49, no. 1, 123-139 (2004).

\bibitem{TT 02}
T. T. Trent, A new estimate for the vector valued corona problem, J. Funct. Anal. 189, no. 1, 267-282 (2002).

\bibitem{X 98}
J. Xiao, The $\overline{\partial}$-problem for multipliers of the Sobolev space, Manuscripta Math. 97, no. 2, 217-232 (1998).
\end{thebibliography}
\end{document}